\documentclass[graybox]{svmult}

\usepackage{type1cm}        
\usepackage{makeidx}         
\usepackage{graphicx}        
\usepackage{multicol}        
\usepackage[bottom]{footmisc}

\usepackage{newtxtext}       %
\usepackage[varvw]{newtxmath}       

\makeindex             

\begin{document}

\title*{Blowups and Tops of Overlapping Iterated Function Systems}
\author{Louisa F. Barnsley and Michael F. Barnsley}
\institute{Louisa F. Barnsley and Michael F. Barnsley \at Australian National University,
Canberra, ACT, Australia \email{michael.barnsley@anu.edu.au}}

\maketitle

\begin{center}
\textit{Dedicated to Robert Strichartz}

\medskip
\end{center}

\abstract{We review aspects of an important paper by Robert Strichartz concerning
reverse iterated function systems (i.f.s.) and fractal blowups. We compare the
invariant sets of reverse i.f.s. with those of more standard i.f.s. and with
those of inverse i.f.s. We describe Strichartz' fractal blowups and explain how
they may be used to construct tilings of $\mathbb{R}^{n}$ even in the case
where the i.f.s. is overlapping. We introduce and establish the notion of
\textquotedblleft tops\textquotedblright\ of blowups. Our motives are not
pure: we seek to show that a simple i.f.s. and an idea of Strichartz, can be
used to create complicated tilings that may model natural structures.}

\begin{figure}[ptb]%
\centering
\includegraphics[
height=2.7657in,
width=3.4203in
]%
{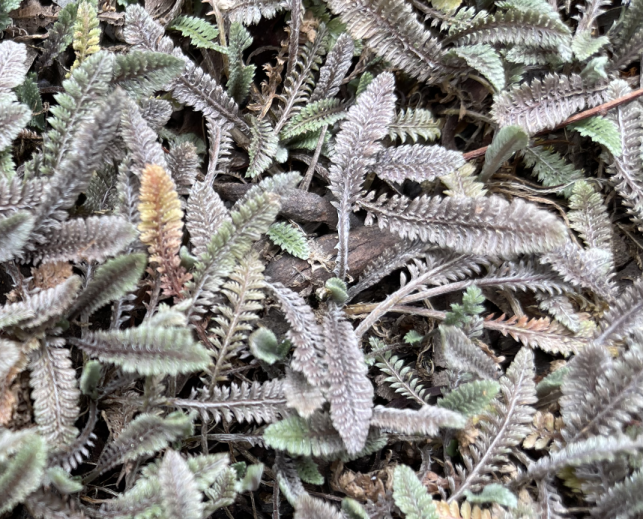}%
\caption{Small ferns growing wildly: can fractal geometry model such images?}%
\label{ferny}%
\end{figure}

\section{Introduction}
In \textquotedblleft Fractals in the large" \cite{strichartz} Robert
Strichartz observes that fractal structure is characterized by repetition of
detail at all small scales. He asks \textquotedblleft Why not large scales as
well?\textquotedblright\ He proposes two ways to study large scaling
structures using developments of iterated function systems. Here we review
geometrical aspects of his paper and make a contribution in the area of tiling theory.

In his first approach, Strichartz defines a \textit{reverse iterated function
system} (r.i.f.s.) to be a set of $m>1$ expansive maps
\[
T=\{t_{i}:M\rightarrow M|i=1,2,\dots,m\}
\]
acting on a locally compact discrete metric space $M$, where every point of
$M$ is isolated. Here the large scaling structures are the invariant sets of
$T$, sets $S\subset M$ which obey
\[
S=\bigcup\limits_{i=1}^{m}t_{i}\left(  S\right)  .
\]

Why does does Strichartz restrict his definition to functions acting on
\textit{discrete} metric spaces? (i) He establishes that there are interesting
nontrivial examples. (ii) He shows that such objects (act as a kind of
skeleton to) play a role in his second kind of large scale fractal structure
that he calls a \textit{fractal blowup. }Probably he had other reasons related
to situations where his approach to analysis on fractals could be explored.

In Section \ref{preliminaries} we present notation for iterated function
systems (i.f.s.) acting on $\mathbb{R}^{n}$. We are particularly concerned
with notation for chains of compositions of functions and properties of
addresses of points on fractals. In Section \ref{reversesec} we review
Strichartz' definition and basic theorem concerning invariant sets of reverse
iterated function systems, and we compare them to the corresponding situation
for contractive i.f.s. We describe some kinds of invariant sets of contractive
i.f.s. and consider how they compare to Strichartz' large scaling structures.
It is a notable feature of Strichartz' definition that he restricts attention
to functions acting on compact discrete metric spaces. We mention that, if
this restriction is lifted, sometimes very interesting structures,
characterized by repetition of structure at large scales, may be obtained. See
for example Figure 2.

In Section \ref{blowupsec} we define fractal blowups, Strichartz' second kind
of large scale fractal structure, and present his characterization of them,
when the open set condition (OSC) is obeyed, as unions of scaled copies of an
i.f.s. attractor, with the scaling restricted to a finite range. We outline
the proof of his characterization theorem using different notation,
anticipating fractal tops. We recall Strichartz' final theorem on the topic,
where he restricts attention to blowups of an i.f.s. all of the same scaling
factor. Here he combines his two ideas: he reveals that the fractal blowup is
in fact a copy of the original fractal translated by all the points on an
invariant set of a r.i.f.s.

In Section \ref{topsec1} we discuss how tilings of blowups can be extended to
overlapping i.f.s. In \cite{tilings} it was shown how, in the overlapping (OSC
not obeyed) case, tilings of blowups can be defined using an artificial
recursive system of \textquotedblleft masks\textquotedblright. Here the
approach is more natural, but we pay a price--sequences of tilings are not
necessarily nested. Here tilings are defined by using \textquotedblleft
fractal tops\textquotedblright, namely attractors with their points labelled
by lexicographically highest addresses. The needed theory of fractal tops is
developed in Subsection \ref{ftopssubsection}. Then in Subsection
\ref{topblowupsec} we use these top addresses to define and establish the
existence of tilings of some blowups for overlapping i.f.s. The main theorem
concerns the relationship between the successive tilings that may be used to
define a tiling of a blowup. In Subsection \ref{leafysec} we present an
example involving a tile that resembles a leaf.

Strichartz' paper has overlap with \cite{bandt2}, published about the same
time by Christoph Bandt. Both papers consider the relationship between i.f.s.
theory and self-similar tiling theory. Current work in tiling theory
does not typically use the mapping point of view, but both Bandt and Stricharz do. Bandt is
particularly focused on the open set condition and the algebraic structure of
tilings, but also has a clear understanding of tilings of blowups when the OSC
is obeyed.

Strichartz' paper also contains measure theory aspects that we do not discuss. But
from the little we have focused on here, much has been learned concerning the
subtlety, the depth, and the elegant simplicity of the mathematical thinking
of Robert Strichartz.

\section{Preliminaries}
\label{preliminaries}
Let $\mathbb{N=\{}1,2,\dots\}.$ An iterated function system (i.f.s.) is a set
of functions
\[
F=\{f_{i}:\mathbb{X\rightarrow X}|i=1,2,\dots,m\}
\]
mapping a space $\mathbb{X}$ into itself, with $m\in\mathbb{N}$. An invariant
set of $F$ is $S\subset\mathbb{X}$ such that
\[
S=F(S):=\bigcup\limits_{i=1}^{m}f_{i}(S)\text{ where }f_{i}(S)=\{f_{i}(s)|s\in
S\}.
\]
We use the same symbol $F$ for the i.f.s. and for its action on $S$, as
defined here. 

The i.f.s. $F$ is said to be contractive when $\mathbb{X}$ is equipped with a
metric $d$ such that $d(f_{i}x,f_{i}y)\leq\lambda d(x,y)$ for some
$0<\lambda<1$ and all $x,y\in\mathbb{X}$. If $\mathbb{X=R}^{n},\ $we take $d$
to be the Euclidean metric. A contractive i.f.s. on $\mathbb{R}^{n}$ is
associated with its attractor $A$, the unique non-empty closed and bounded
invariant set of $F$ \cite{hutchinson}. But Strichartz is also interested in
the case where the underlying space is discrete and the maps are expansive.

We use \textit{addresses} to describe compositions of maps. Addresses are
defined in terms of the indices of the maps of $F$. Let $\Sigma=\left\{
1,2,\dots,m\right\}  ^{\mathbb{N}}$, the set of strings of the form
$\mathbf{j}=j_{1}j_{2}\dots$ where each $j_{i}$ belongs to $\left\{
1,2,\dots,m\right\}  $. We write $\Sigma_{n}=\left\{  1,2,\dots,m\right\}
^{n}$ and let $\Sigma_{\mathbb{N}}=\cup_{n=1}^{\infty}\Sigma_{n}.$ The address
$\mathbf{j}\in\Sigma$ truncated to length $n$ is denoted by $\mathbf{j}%
|n=j_{1}j_{2}\dots j_{n}\in\Sigma_{\mathbb{N}},$ and we define
\begin{align*}
f_{\mathbf{j}|n}  &  =f_{j_{1}}f_{j_{2}}\dots f_{j_{n}}=f_{j_{1}}\circ
f_{j_{2}}\circ\dots\circ f_{j_{n}},\\
f_{-\mathbf{j}|n}  &  =f_{j_{1}}^{-1}f_{j_{2}}^{-1}\dots f_{j_{n}}%
^{-1}=f_{j_{1}}^{-1}\circ f_{j_{2}}^{-1}\circ\dots\circ f_{j_{n}}^{-1}.
\end{align*}
Define a metric $d$ on $\Sigma$ by $d(\mathbf{j},\mathbf{k})=2^{-\max
\{n|j_{m}=k_{m},m=1,2,...,n\}}$ for $\mathbf{j}\neq\mathbf{k}$, so that
$\left(  \Sigma,d\right)  $ is a compact metric space.

The \textit{forward orbit }of a point $x$ under (the semigroup generated by)
$F$ is
\[
\{f_{\mathbf{j}|n}(x)|\mathbf{j}\in\Sigma,n\in\mathbb{N\}}\text{.}%
\]
Here we do not allow $\mathbf{j}|n$ to be the empty set, so $x$ is not
necessarily an element of its forward orbit under the i.f.s. Indeed, $x$ is a
member of its forward orbit if and only if $x$ is a fixed point of one of the
composite maps $f_{\mathbf{j}|n}$.

Now let $F$ be a contractive IFS of invertible maps on $\mathbb{R}^{n}$. Then
a continuous surjection $\pi:\Sigma\rightarrow A$ is defined by%
\[
\pi(\mathbf{j})=\lim_{N\rightarrow\infty}f_{\mathbf{j}|N}(x)=\lim
_{N\rightarrow\infty}f_{j_{1}}f_{j_{2}}\dots f_{j_{N}}(x).
\]
The limit is independent of $x$. The convergence is uniform in $\mathbf{j}$
over $\Sigma,$ and uniform in $x$ over any compact subset of $\mathbb{R}^{n}$.
We say $\mathbf{j}\in\Sigma$ is an \textit{address} of the point
$\pi(\mathbf{j})\in A$.

We define $i:\Sigma\rightarrow\Sigma$ by $i(\mathbf{j)=}ij_{1}j_{2}\dots$ But
we may also write $k_{1}k_{2}...k_{l}\mathbf{j}$ to mean the address
$k_{1}k_{2}\dots k_{l}j_{1}j_{2}\dots\in\Sigma.$ Let $\sigma:\Sigma
\rightarrow\Sigma$ be the shift operator defined by $\sigma(\mathbf{j)=}%
j_{2}j_{3}\dots$ . It is well-known that%
\[
f_{i}\circ\pi=\pi\circ i\mathbf{\ }\text{and }\pi\circ\sigma\left(
\mathbf{j}\right)  =f_{j_{1}}^{-1}\circ\pi\left(  \mathbf{j}\right)
\]
for all $i\in\{1,2,...m\},\mathbf{j\in}\Sigma$.

A notable shift invariant subset of $\Sigma$ is the set of disjunctive addresses $\Sigma_{dis}$. An address $\mathbf{j}\in\Sigma$ is \textit{disjunctive} when,
for each finite address $i_{1}i_{2}i_{3}\dots i_{k}\in\left\{  1,2,\dots
,m\right\}  ^{k}$, there is $l\in\mathbb{N}$ so that $j_{l+1}...j_{l+k}%
=i_{1}i_{2}i_{3}\dots i_{k}$. The
set of disjunctive addresses $\Sigma_{dis}\subset\Sigma$ is totally invariant
under the shift, according to $\sigma\left(  \Sigma_{dis}\right)
=\Sigma_{dis}.$  A point $a\in A$ is disjunctive if there is a
disjunctive address $\mathbf{j}\in\Sigma$ such that $\pi(\mathbf{j})=a$. Disjunctive points play a role in the structure of attractors.
For example, if\ the i.f.s. obeys the open set condition (OSC) and its
attractor has non-empty interior, then all the disjunctive points belong to
the interior of the attractor \cite{bandt}. Recall that $F$ obeys the OSC when
there exists a nonempty open set $O$ such that $\cup f_{i}(O)\subset O$ and
$f_{i}(O)\cap f_{j}(O)=\emptyset$ whenever $i\neq j.$

\section{Reverse iterated function systems}
\label{reversesec}
In his first approach to large scaling structures, Strichartz defines a
\textit{reverse iterated function system} (r.i.f.s.) to be a set of $m>1$
expansive maps
\[
T:=\{t_{i}:M\rightarrow M|i=1,2,\dots,m\}
\]
acting on a locally compact discrete (i.e. every point is isolated) metric
space $M$.\textit{ }We write $T$ and $t_{i}$ in place of $F$ and $f_{i}$ to
distinguish this special kind of i.f.s. A mapping $t_{i}:M\rightarrow M$ is
said to be expansive if there is a constant $r>1$ such that $d(t_{i}%
x,t_{i}y)\geq rd(x,y)$ for all $x\neq$ $y$ in $M.$ An expansive mapping is
necessarily one-to-one and has at most one fixed point.

In this case Strichartz's large scaling structures are the invariant sets of
r.i.f.s.; that is, sets $S\subset M$ which obey
\[
S=T(S)=\bigcup\limits_{i=1}^{m}t_{i}\left(  S\right)  .
\]
By requiring that $M$ is discrete, Strichartz restricts the possible invariant sets to be discrete.

Let $P$ be the fixed points of $\{t_{\mathbf{i}|k}:M\rightarrow M|k\in
\mathbb{N},\mathbf{i}\in\Sigma\}$. Contrast Theorem \ref{theorem2} with
Theorem \ref{theorem2x}.

\begin{theorem}
[Strichartz]\label{theorem2}A set is invariant for a r.i.f.s. if and only if
it is a finite union of forward orbits of points in $P$. In particular,
invariant sets exist if and only if $P$ is nonempty, and there are at most a
finite number of invariant sets.
\end{theorem}

EXAMPLE 1 Let
$M=\mathbb{Z}$, $T=\{t_{i}:M\rightarrow M;t_{1}(x)=2x,t_{2}(x)=2x-1\}.$ 
It is readily verified that $M$ is invariant for this r.i.f.s, $T.$ It consists of
the forward orbits of the fixed points of $t_{1}$ and $t_{2}.$

EXAMPLE 2
Strichartz presents the following example of a r.i.f.s. Let $M$ be the set of
integer lattice points $\mathbb{Z}^{2}$ in the plane, lying between or on the
lines $y=\rho x$ and $y=\rho x+1$ where $\rho+\rho^{2}=1,$ $\rho=\left(
\sqrt{5}-1\right)  /2.$ The r.i.f.s. comprises the two maps%
\[
t_{1}(x,y)=(-x-y,-x),t_{2}(x,y)=(1-x-y,1-x).
\]
These maps are expansive on $M$, even though when viewed as transformations
acting on $\mathbb{R}^{2},$ they contract pairs of points that lie on any
straight line with slope $-1/\rho$. The fixed point of $t_{1}$ is $(0,0)$ and
of $t_{2}$ is $(0,1),$ both of which lie in $M.$ The union of the forward
orbits of these two points is $M$. So this unlikely looking set of discrete
points is invariant under the r.i.f.s.

This example yields, by projection onto the line $y=\rho x,$ an example of a
quasi-periodic linear tiling using tiles of lengths $\rho$ and $1+\rho.$
Strichartz also points out that by projection onto the perpendicular line
$y=-x/\rho$ of a natural measure on $M$ one obtains, after renormalizing, the
unique self-similar measure on $[0,1]$ associated with the overlapping i.f.s.
$f_{1}(x)=\rho x$, $f_{2}(x)=\rho x+1$ with equal probabilities.

Theorem \ref{theorem2} leads one to wonder: What are the invariant sets of an
i.f.s.? Usually the focus is on compact invariant sets, namely attractors. The
following Theorem is simply a list of some of the invariant sets of a
contractive i.f.s. The wealth of such invariants here stands in sharp contrast
to Theorem \ref{theorem2}.

\begin{theorem}
[Some Invariant Sets of an i.f.s.]\label{theorem2x}Let $F$ be a contractive
i.f.s. of invertible maps on $\mathbb{R}^{n}.$ If $S\subset\mathbb{R}^{n}$ is
invariant and bounded, then either $\overline{S}=\emptyset,$ or $\overline
{S}=A.$ The followings sets are invariant.

\begin{enumerate}
\item The attractor $A,$ and the whole space $\mathbb{R}^{n}.$

\item The forward orbit under $F$ of any periodic point $p\in P.$

\item The set of disjunctive points of $A$.

\item The orbit of any $x\in\mathbb{R}^{n}$ under the free group generated by
the maps of $F$ and their inverses.

\item The union of any collection of invariant sets.
\end{enumerate}
\end{theorem}

There are other invariant sets. For example, let $A$ be a Sierpinski triangle,
the attractor of an i.f.s. $F_{sierp}$ in the usual way. Let $B$ be the union
of the sides of all triangles in $A.$ Then $B$ is an invariant set for
$F_{sierp}$. It is not covered by Theorem \ref{theorem2x}.

We note that the invariant set in (4) is also invariant under the
\textit{inverse i.f.s.}
\[
F^{-1}:=\left\{  f_{i}^{-1}:\mathbb{R}^{n}\rightarrow\mathbb{R}^{n}%
|i=1,\dots,m\right\}  .
\]

The orbit under $F^{-1}$ of the attractor $A$ is invariant under
$F^{-1}.$ This set may be referred to as the \textit{fast basin} of $A$ with respect
to $F$, see \cite{fastbasin}. It is an example of a set which is
\textquotedblleft invariant in the large\textquotedblright, admitted when
Strichartz' constraint, that the underlying space is discrete and locally
compact, is lifted.

Figure \ref{fastbasin} illustrates the fast basin associated with (left) a
Sierpinski triangle i.f.s. and (right) a different i.f.s. of three similitudes
of scaling factor $1/2.$ Fast basins are interesting from a computational
point of view, because they comprise the points $x$ in $\mathbb{R}^{n}$ for
which there is an address $\mathbf{j\in}\Sigma_{\mathbb{N}}$ such that
$f_{\mathbf{j}}(x)\in A$.

\begin{figure}[ptb]%
\centering
\includegraphics[
height=2.2in,
width=4.4529in
]%
{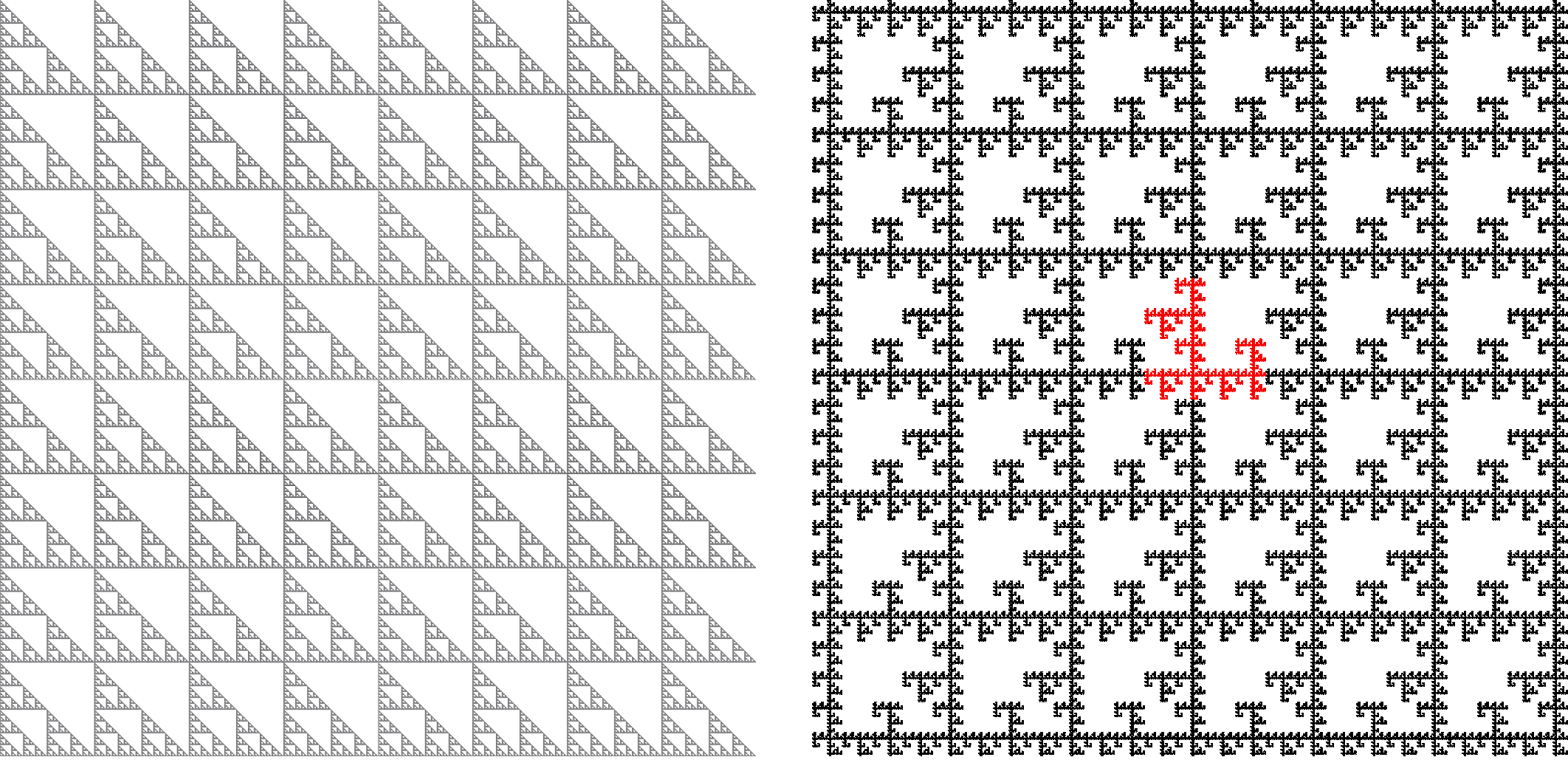}%
\caption{Two examples of invariant sets of inverse iterated function systems. The left image illustrates part of the fast basin of a Sierpinski triangle i.f.s. The right image illustrates the fast basin of an i.f.s. whose attractor is illustrated in red. These unbounded  sets are "invariant in the large" but are not discrete.}%
\label{fastbasin}%
\end{figure}

\section{\label{blowupsec}Strichartz' fractal blowups}

Strichartz uses r.i.f.s. to analyze the structure of what he christened
\textquotedblleft fractal blowups\textquotedblright. These structures have
been used to develop differential operators on unbounded fractals, see for
example \cite{strichartz2, teplyaev}.

Let $F$ be an i.f.s. of similitudes. The maps take the form
\[
f_{j}\left(  x\right)  =r_{j}U_{j}x+b_{j}%
\]
where $0<r_{j}<1,$ $b_{j}\in\mathbb{R}^{n}$ and $U_{j}$ is an orthogonal
transformation. It is convenient to write $r_{j}=r^{a_{j}}$ where
$r=\max\left\{  r_{j}\right\}  $, so that $1\leq a_{j}<a_{\max}$ . A
\textit{blowup }$\mathcal{A}$ of $A$ is the union of an increasing sequence of
sets
\begin{equation}
A=A_{0}\subset A_{1}\subset A_{2}\subset\dots\label{equationA}%
\end{equation}
where $A_{j}=f_{-\mathbf{k}|j}\left(  A\right)  $ for some fixed
$\mathbf{k}\in\Sigma$ and all $j\in\mathbb{N}$. We have
\begin{equation}
\mathcal{A=A}\left(  \mathbf{k}\right)  =\bigcup\limits_{j=1}^{\infty
}f_{-\mathbf{k}|j}\left(  A\right)  . \label{equation1}%
\end{equation}
Strichartz starts with a more general definition of a blowup, but restricts
consideration to the one given here.

\begin{theorem}
[Strichartz]\label{strichtheorem}Let $\mathcal{A}\left(  \mathbf{k}\right)  $
be a blowup of $A$ of the form in Equation (\ref{equation1}) and assume $F$
satisfies the OSC. Then $\mathcal{A}\left(  \mathbf{k}\right)  $ is the union
of sets $\mathcal{G}_{n}$ which are similar to $A$ with the contraction ratios
bounded from above and below, and the number of sets $\mathcal{G}_{n}$ that
intersect any ball of radius $R$ is at most a multiple of $R^{n}.$ In
particular the union $\mathcal{A}\left(  \mathbf{k}\right)  =\cup
_{n=1}^{\infty}\mathcal{G}_{n}$ is locally finite, and the intersection of
$\mathcal{A}\left(  \mathbf{k}\right)  $ with any compact set is equal to the
intersection of $\cup_{n=1}^{N}\mathcal{G}_{n}$ with that compact set for
large $N$.
\end{theorem}

\begin{proof}
We outline a proof for the case of a single scaling factor $0<r<1\ $with
$f_{i}(x)=rU_{i}x+b_{i}$. At heart, our proof is the same as Strichartz, but
we introduce notation that helps with our generalization to overlapping i.f.s.
in Section \ref{topsec1}.

Since $F$ satisfies the OSC, there is a bounded open set $\mathcal{O}$ such
that $A\subset$ $\overline{\mathcal{O}},$ $f_{i}(\mathcal{O)\subset O}$ for
all $i,$ $f_{i}(\mathcal{O)\cap}$ $f_{j}(\mathcal{O)=\emptyset}$ for all
$i\neq j.$ Note that the latter condition implies that the sets in $\left\{
f_{j_{1}j_{2}...j_{l}}(\mathcal{O)}|j_{1}j_{2}...j_{l}\in\Sigma_{l}\right\}  $
are disjoint.

Define a collection of sets%
\[
\Pi_{S}(\mathbf{k|}n):=\left\{  f_{-\mathbf{k}|n}f_{\mathbf{m|}n}%
(S)|\mathbf{m}\in\Sigma\right\}
\]
where $S$ may be $\mathcal{O}$, $\overline{\mathcal{O}},$ or $A.$ Observe
that
\[
\Pi_{A}(\mathbf{k|}1)\subset\Pi_{A}(\mathbf{k|}2)\subset...
\]
and
\[
f_{-\mathbf{k}|l}\left(  A\right)  =\bigcup\limits_{n=1}^{\l}\Pi
_{A}(\mathbf{k|}n).
\]
Also
\[
\Pi_{\mathcal{O}}(\mathbf{k|}\left(  n+1\right)  )\backslash\Pi_{\mathcal{O}%
}(\mathbf{k|}n)=\left\{  f_{-\mathbf{k}|n+1}f_{\mathbf{m|}\left(  n+1\right)
}(\mathcal{O})|\mathbf{m}\in\Sigma_{n+1},k_{n+1}\neq m_{1}\right\}
\]
consists of $m^{n}(m-1)$ disjoint open sets. It follows that $\left\{
\Pi_{\mathcal{O}}(\mathbf{k|}n)|n=1,2,\dots\right\}  $ is a nested increasing
sequence of disjoint open sets, whose closed union contains $\mathcal{A}%
\left(  \mathbf{k}\right)  $. The closure of each open set contains a copy of
$A.$ Since each open set has volume bounded below by a positive constant,
local finiteness is assured.

A general case of a Strichartz style blowup is captured by defining
\begin{align*}
\Pi_{S}\left(  \mathbf{k|}j\right)   &  =f_{-\mathbf{k}|l}(\{f_{\mathbf{m}%
}(S)|\eta^{-}(\mathbf{m})<\eta(\mathbf{k}|l)\leq\eta(\mathbf{m}),\mathbf{m}%
\in\Sigma_{\mathbb{N}}\})\\
\Pi_{S}\left(  \mathbf{k}\right)   &  =\bigcup\limits_{j\in\mathbb{N}}\Pi
_{S}\left(  \mathbf{k|}j\right)
\end{align*}
where
\begin{align*}
\eta^{-}(m_{1}m_{2}\dots m_{n}) &  =a_{m_{1}}+a_{m_{2}}+\dots+a_{n-1}\\
\eta(m_{1}m_{2}\dots m_{n}) &  =a_{m_{1}}+a_{m_{2}}+\dots+a_{n}%
\end{align*}
These formulas provide a specific form to Strichartz' stopping time argument.
Using these more general expressions one obtains, for fixed $\mathbf{k}$, an
increasing union of copies of $A$ scaled by factors that lie between
$r^{a_{\max}}$ and $r.$ See for example \cite{tilings, polygon}. The argument
concerning local finiteness is essentially the same as above.
\end{proof}

Strichartz unites his two ideas, reverse i.f.s. and blowups, by considering
the case where $f_{j}\left(  x\right)  =rx+b_{j}$ for all $j,$ and studying
the blowup $\mathcal{A}\left(  \overline{\mathbf{1}}\right)  $ where
$\overline{\mathbf{1}}\mathbf{=}$ $111\dots,$ that is
\[
\mathcal{A}\left(  \overline{\mathbf{1}}\right)  =\cup_{n=1}^{\infty}%
(f_{1}^{-1})^{n}A
\]

\begin{theorem}
[Strichartz combines r.i.f.s. and blowups]\label{theorem4}Let $f_{j}%
x=rx+b_{j}.$ Then $\mathcal{A}\left(  \overline{\mathbf{1}}\right)  =A+D$
where $D$ is an invariant set for the r.i.f.s.
\[
t_{j}\left(  x\right)  =r^{-1}(x+b_{j}-b_{1}),j=1,2,...,m
\]
Specifically, $D$ is the forward orbit of $0$, the fixed point of $t_{1}.$
\end{theorem}

That is, $\mathcal{A}\left(  \overline{\mathbf{1}}\right)  $ is the Minkowski
sum of the attractor of the i.f.s. and an invariant set of a r.i.f.s.

\section{\label{topsec1}Tops tilings}

In this Section we study tilings of fractal blowups in the case of overlapping
i.f.s. attractors. First, in Subsection \ref{ftopssubsection} we give relevant
theory of fractal tops. In Subsection \ref{topblowupsec} we show how fractal
tops may be used to generate tilings of fractal blowups for overlapping i.f.s.
The approach here is distinct from the one in \cite{tilings}. In Subsection
\ref{leafysec} we illustrate fractal tops for an i.f.s. of two maps, with
overlapping attractor that looks like a leaf, suggesting applications to
modelling of complicated real-world images.

\subsection{Fractal tops\label{ftopssubsection}}

Let $F$ be a strictly contractive i.f.s. acting on a complete metric space
$\mathbb{X}$, with maps $f_{i}$ and attractor $A.$ We assume that there are
two or more maps, at least two of which have different fixed points. Also all
of the maps are invertible.

\begin{lemma}
\label{lemma1}Let $C$ be a closed subset of $\Sigma.$ Let $\mathbf{j=}$
$\max\{\mathbf{k}\in C\}.$ Then $\mathbf{j}=j_{1}\max\{\mathbf{m}\in
\Sigma|(j_{1}\mathbf{m)\in}C\}.$
\end{lemma}

\begin{proof}
$C$ is the union of the three closed sets $\{\mathbf{k}\in C|k_{1}>j_{1}%
\}$,$\{\mathbf{k}\in C|k_{1}=j_{1}\},$ and $\{\mathbf{k}\in C|k_{1}<j_{1}\}.$
The maximum over $C$ is the maximum of the maxima over these three sets. But
the set $\{\mathbf{k}\in C|k_{1}>j_{1}\}$ is empty, because if not then
$\max\{\mathbf{k}\in C\}\geq\max\{\mathbf{k}\in C|k_{1}>j_{1}\}>\mathbf{j}$
which is a contradiction. If $\max\{\mathbf{k}\in C\}=\max\{\mathbf{k}\in
C|k_{1}<j_{1}\}$ then $\mathbf{j>j}$, again a contradiction.
\end{proof}

Since $\pi:\Sigma\rightarrow A$ is continuous and onto, it follows that
$\pi^{-1}(x)$ is closed for all $x\in A.$ Lemma \ref{lemma1} tells us that a
map $\tau:A\rightarrow\Sigma$ and subset $\Sigma_{top}\subset\Sigma$ are
well-defined by
\[
\tau(x):=\max\{\mathbf{k\in}\Sigma|\pi(\mathbf{k)=}x\},\Sigma_{top}:=\tau(A).
\]
Conventionally the maximum here is with respect to lexicographical ordering.
We refer loosely to these objects and the ideas around them as relating to the
\textit{top }of $A.$ Formally, the top of $A$ is the graph of $\tau,$ namely
$\left\{  (x,\tau(x))|x\in A\right\}  $.

Top addresses of points in $A,$ namely points in $\Sigma_{top},$ can be
calculated by following the orbits of the shift map $\sigma:\Sigma
_{top}\rightarrow\Sigma_{top}.$ Simply partition $A$ into $A_{1}=f_{1}(A),$
$A_{2}=f_{2}(A)\backslash A_{1},$\ $A_{3}=f_{3}(A)\backslash(A_{1}$ $\cup
A_{2}),\dots,A_{m}=f_{m}(A)\backslash\cup_{n\neq m}A_{n}$. Define the orbit
$\left\{  x_{n}\right\}  _{n=1}^{\infty}$ of $x=x_{1}\in A,$ under the tops
dynamical system, by $x_{n+1}=f_{i_{n}}^{-1}(x_{n})$ where $i_{n}$ is the
unique index such that $x_{n}\in A_{i_{n}}$.

A version of the following observation can be found in \cite{tops}. See also
\cite{bandt}.

\begin{theorem}
\label{proposition1}The set of top addresses is shift invariant, according to
$\Sigma_{top}=\sigma\left(  \Sigma_{top}\right)  $ where $\sigma$ is the left shift.
\end{theorem}

\begin{proof}
First we show that $\sigma\left(  \tau(A)\right)  \subset\tau(A).$ If\textbf{
}$\mathbf{j}\in\tau(A),$ then%
\begin{align*}
\mathbf{j}  &  =\max\{\mathbf{k}\in\Sigma|\pi(\mathbf{k})=\pi(\mathbf{j}%
)\}\text{(by definition)}\\
&  =\max\{j_{1}\mathbf{l}\in\Sigma|\pi(j_{1}\mathbf{l})=\pi(\mathbf{j}%
)\}\text{ (by Lemma \ref{lemma1})}\\
&  =j_{1}\max\{\mathbf{l}\in\Sigma|f_{j_{1}}(\pi(\mathbf{l}))=f_{j_{1}}%
(\pi(\sigma\mathbf{j}))\}\\
&  =j_{1}\max\{\mathbf{l}\in\Sigma|\pi(\mathbf{l})=\pi(\sigma\left(
\mathbf{j}\right)  )\}\text{ (since }f_{j_{1}}\text{ is invertible)}\\
&  =j_{1}\tau(\pi(\sigma\left(  \mathbf{j}\right)  )).
\end{align*}
Hence $\sigma(\mathbf{j})=\tau(\pi(\sigma\left(  \mathbf{j}\right)  )).$ Hence
$\left\{  \sigma\left(  \mathbf{j}\right)  \mathbf{|j}\in\tau(A)\right\}
=\left\{  \tau(\pi(\sigma\left(  \mathbf{j}\right)  ))\mathbf{|j}\in
\tau(A)\right\}  $ which implies $\sigma\left(  \tau(A)\right)  \subset
\tau(A).$

We also have $1\left(  \Sigma\right)  \subset\Sigma$ so $\tau(\pi(1\left(
\Sigma\right)  ))\subset\tau(\pi(\Sigma))=\tau(A).$ But $\tau(\pi(1\left(
\Sigma\right)  )))=1\left(  \tau(\pi(\Sigma))\right)  $ by a similar argument
to the proof of Lemma \ref{lemma1}, so $1\tau(A)\subset\tau(A).$ Applying
$\sigma$ to both sides, we obtain $\tau(A)\subset\sigma\left(  \tau(A)\right)
.$
\end{proof}

It appears that the shift space $\Sigma_{top}$ is not of finite type in general, and graph directed constructions cannot be used in general.
This is a topic of ongoing research.

Define $\Sigma_{top,n}$ to be the elements of $\Sigma_{top}$ truncated to the
first $n$ elements. That is,
\[
\Sigma_{top,n}=\{\left(  \mathbf{k}|n\right)  |\mathbf{k\in}\Sigma_{top}\}.
\]
Define $\pi_{top}:$ $\Sigma_{top}\rightarrow A$ to be the restriction of
$\pi:\Sigma\rightarrow A$ to $\Sigma_{top}.$ Extend the definition of
$\pi_{top}$ so that it acts on truncated top addresses according to:
\[
\pi_{top}\left(  \mathbf{k|}n\right)  =\{x\in f_{\mathbf{k}|n}(A)|x\notin
f_{\mathbf{c}|n}(A)\text{ for all }\mathbf{c}|n>\mathbf{k}|n\}
\]
for all $\mathbf{k\in}\Sigma_{top}$ and all $n\in\mathbb{N}$. We will make use
of the following observation.

\begin{lemma}
\label{topshiftlemma}If $\mathbf{k}\in\Sigma_{top},$ then $f_{k_{1}}\left(
\pi_{top,n-1}(\sigma\left(  \mathbf{k|}n\right)  )\right)  \supset$
$\pi_{top,n}(\mathbf{k|}n)$ for all $n\in\mathbb{N}$.
\end{lemma}

\begin{proof}
We need to compare the sets
\[
\{f_{k_{1}...k_{n}}(x)|f_{k_{1}...k_{n}}(x)\notin f_{l_{1}l_{2}...l_{n}%
}(A)\text{ for all }l_{1}...l_{n}>k_{1}...k_{n}\}
\]
and
\[
\{f_{k_{1}...k_{n}}(x)|f_{k_{1}}f_{k_{2}...k_{n}}(x)\notin f_{k_{1}}%
f_{l_{2}...l_{n}}(A)\text{ for all }l_{2}...l_{n}>k_{2}...k_{n}\}.
\]
The condition in the latter expression is less restrictive.
\end{proof}

The sets of truncated top addresses $\Sigma_{top,n}$ have an interesting
structure. Any addresses in $\Sigma_{top,n}$ can be truncated on the left or
on the right to obtain an address in $\Sigma_{top,n-1}.$ The following Lemma
is readily verified.

\begin{lemma}
Let $n>1.$ If $i_{1}i_{2}\dots i_{n-1}i_{n}\in\Sigma_{top,n}$, then both
$i_{2}\dots i_{n-1}i_{n}\ $and $i_{1}i_{2}\dots i_{n-1}\ $belong to
$\Sigma_{top,n-1}.$
\end{lemma}

\subsection{Top blowups and tilings\label{topblowupsec}}

Here we are particularly interested in the overlapping case, where the
OSC$\ $does not hold. We show that natural partitions of fractal blowups, that
we call tilings, may still be obtained.

Throughout this subsection, $F$ is an i.f.s. with
\begin{equation}
f_{j}\left(  x\right)  =rU_{j}x+b_{j} \label{ifstype}%
\end{equation}
where $b_{j}\in\mathbb{R}^{n}$ and $U_{j}$ is an orthogonal transformation. We
assume that there are two or more maps, at least two of which have distinct
fixed points. We have in mind the situation where $A$ is homeomorphic to a
ball, although this is not required by Theorems \ref{strichgentheorem} and
\ref{onebartheorem}.

As in Section \ref{blowupsec}, but restricted to $\mathbf{i}\in\Sigma_{top},$
fractal blowups are well defined by\textit{ }%
\[
\mathcal{A}_{n}=\mathcal{A}\left(  \mathbf{i|}n\right)  =\bigcup
\limits_{l=1}^{n}f_{-\mathbf{i}|l}\left(  A\right)  \text{ and }%
\mathcal{A=A}\left(  \mathbf{i}\right)  =\bigcup\limits_{l=1}^{\infty
}f_{-\mathbf{i}|l}\left(  A\right)  .
\]
The unions are of increasing nested sequences of sets so $\mathcal{A}%
_{n}=f_{\mathbf{i}|n}^{-1}\left(  A\right)  $ and $\mathcal{A}=\cup
\mathcal{A}_{n}.$ Note that $\mathcal{A}\left(  \mathbf{i|}n\right)  $ is
related to $\mathcal{A}\left(  \mathbf{j|}n\right)  $ by the isometry $\left(
f_{-\mathbf{j}|n}\right)  \left(  f_{-\mathbf{i}|n}\right)  ^{-1}$. But
possible relationships between $\mathcal{A}\left(  \mathbf{i}\right)  $ and
$\mathcal{A}\left(  \mathbf{j}\right)  $ are quite subtle because inverse
limits are involved.

Under conditions on $F$ and $\mathbf{i}$, stated in Theorems
\ref{strichgentheorem} and \ref{onebartheorem}, we can define a generalized
tilings of $\mathcal{A}\left(  \mathbf{i}\right)  $ with the aid of the
following two definitions:%

\begin{align*}
\Pi_{top}(\mathbf{i}|k) &  :=\left\{  f_{-\mathbf{i}|k}\left(  \left\{
x\in\pi_{top}\left(  \mathbf{t}|\left(  k+1\right)  \right)  \right\}
\right)  |\mathbf{t\in}\Sigma_{top}\right\}  ,\\
\Pi_{top}(\mathbf{i}) &  :=\lim_{k\rightarrow\infty}\Pi_{top}(\mathbf{i}%
|k),\text{ when this limit is well defined.}%
\end{align*}
For example, the limit is well defined when $\Pi_{top}(\mathbf{i}|k)\subset
\Pi_{top}(\mathbf{i}|k+1)$ for all $k,$ as occurs when the OSC holds. As we
will show, it is also well defined in some more complicated situations.

We call each set $f_{-\mathbf{i}|k}\left(  \left\{  x\in\pi_{top}\left(
\mathbf{t}|k\right)  \right\}  \right)  $ a \textit{tile}, and we call the
collection of disjoint sets $\left\{  f_{-\mathbf{i}|k}\left(  \left\{
x\in\pi_{top}\left(  \mathbf{t}|k\right)  \right\}  \right)  |\mathbf{t\in
}\Sigma_{top}\right\}  $ a \textit{partial tiling. }The partial tilings
$\left\{  f_{-\mathbf{i}|k}\left(  \left\{  x\in\pi_{top}\left(
\mathbf{t}|k\right)  \right\}  \right)  |\mathbf{t\in}\Sigma_{top}\right\}  $
are well defined. However, $\Pi_{top}(\mathbf{i})$ may not be well defined,
because there may not be any simple relationship between successive partial
tilings. But when it is well defined, we call it a tiling.

The tiles in the partial tiling $\left\{  f_{-\mathbf{i}|k}\left(  \left\{
x\in\pi_{top}\left(  \mathbf{t}|k\right)  \right\}  \right)  |\mathbf{t\in
}\Sigma_{top}\right\}  $ may be referred to by their addresses. It is
convenient to define
\[
tile(i_{1}i_{2}...i_{k}.t_{1}t_{2}...t_{k})=f_{-\mathbf{i}|k}\left(  \left\{
x\in\pi_{top}\left(  \mathbf{t}|k\right)  \right\}  \right)
\]
for all $\mathbf{i}|k$ and all $\mathbf{t}|k\in\Sigma_{top}.$ We also define
$tile(\varnothing)=A,$ corresponding to $k=0.$

\begin{lemma}
\label{containlemma}This concerns the sequence of tilings $\Pi_{top}%
(\mathbf{i}|n).$ If $i_{n}p_{1}p_{2}...p_{n-1}$ $\in\Sigma_{top,n},$ then
$tile(i_{1}i_{2}...i_{n-1}.p_{1}p_{2}...p_{n-1})\subset tile(i_{1}%
i_{2}...i_{n}.j_{1}j_{2}...j_{n})$ implies $i_{n}p_{1}p_{2}...p_{n-1}%
=j_{1}j_{2}j_{3}..j_{n}$.
\end{lemma}

\begin{proof}
Suppose $tile(i_{1}i_{2}...i_{n-1}.p_{1}p_{2}...p_{n-1})\subset tile(i_{1}%
i_{2}...i_{n}.j_{1}j_{2}...j_{n}).$ Then applying $\left(  f_{-\mathbf{i}%
|\left(  n-1\right)  }\right)  ^{-1}$ to both sides we obtain $\pi
_{top,n-1}(p_{1}p_{2}...p_{n-1})\subset f_{i_{n}}^{-1}(\pi_{top,n}(j_{1}%
j_{2}...j_{n}))$ which is equivalent to
\[
f_{i_{n}}\pi_{top,n-1}(p_{1}p_{2}...p_{n-1})\subset\pi_{top,n}(j_{1}%
j_{2}...j_{n}).
\]
But $\pi_{top,n}(i_{n}p_{1}p_{2}...p_{n-1})\subset f_{i_{n}}\pi_{top,n-1}%
(p_{1}p_{2}...p_{n-1})$ by Lemma \ref{topshiftlemma}, so
\[
\pi_{top,n}(i_{n}p_{1}p_{2}...p_{n-1})\subset\pi_{top,n}(j_{1}j_{2}...j_{n}).
\]
This implies $i_{n}p_{1}p_{2}...p_{n-1}=j_{1}j_{2}j_{3}..j_{n}$ because
otherwise $\pi_{top,n}(i_{n}p_{1}p_{2}...p_{n-1})$ and $\pi_{top,n}(j_{1}%
j_{2}...j_{n})$ are disjoint subsets of $A$.
\end{proof}

We say that $\mathbf{i\in}\Sigma_{top}$ is \textit{reversible} when, for each
$n\in\mathbb{N}$ there exists $\mathbf{j=j}_{n}\mathbf{\in}\Sigma_{top}$ such
that $j_{1}=i_{n},j_{2}=i_{n-1},...,j_{n}=i_{1}.$ Note that $\mathbf{j}$
depends on $n.$ The address $\overline{1}=11111\dots$ is reversible and belongs to $\Sigma
_{top}$ in all cases.

EXAMPLE 3 For the i.f.s. $\left\{  \mathbb{R};f_{1}(x)=2x/3;f_{2}%
(x)=2x/3+1/3\right\}  ,$ the strings $\overline{1}$ and $\overline{2}$ both
belong to $\Sigma_{top}$ and are reversible. Figure \ref{topaddresses3x} and
Figure \ref{topaddressessq - copy (2)} illustrate two ways of looking at the
development of top addresses. Figure \ref{topadd} (a) illustrates the sets in
$\Sigma_{top,n}$ for $n=0,1,2,3,4,5$. We usually use lexicographic ordering to
define top addresses, but Figure \ref{topaddressessq - copy (2)} uses standard ordering.

\begin{figure}[ptb]%
\centering
\includegraphics[
height=2.7657in,
width=5.5in
]%
{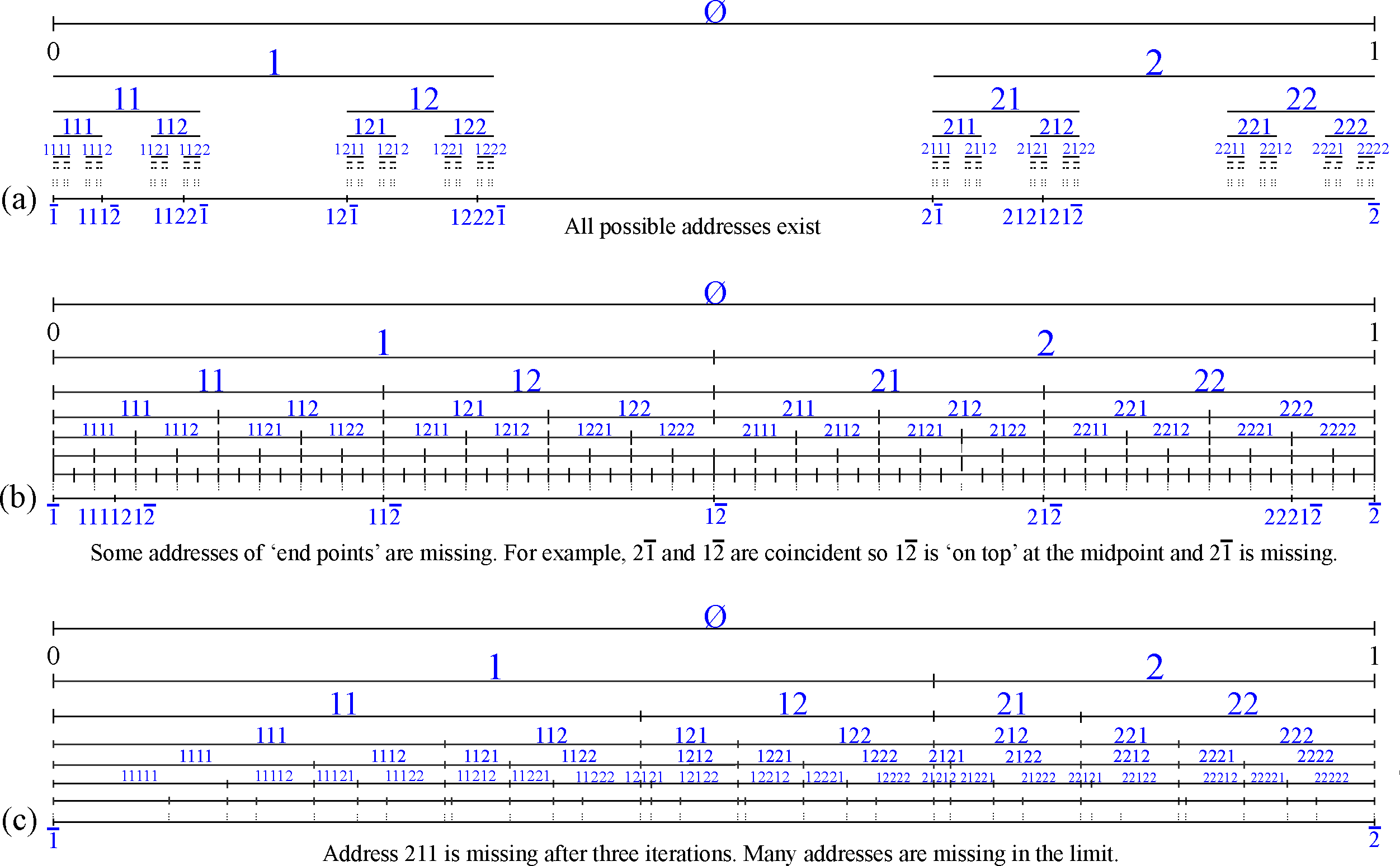}%
\caption{This compares the development of the top
addresses for an i.f.s. of two maps in the cases (a) where each scaling is 1/3 (b) each scaling is  1/2 (c) each scaling is 2/3.}%
\label{topaddresses3x}%
\end{figure}

\begin{figure}[ptb]%
\centering
\includegraphics[
height=3.0415in,
width=3.0381in
]%
{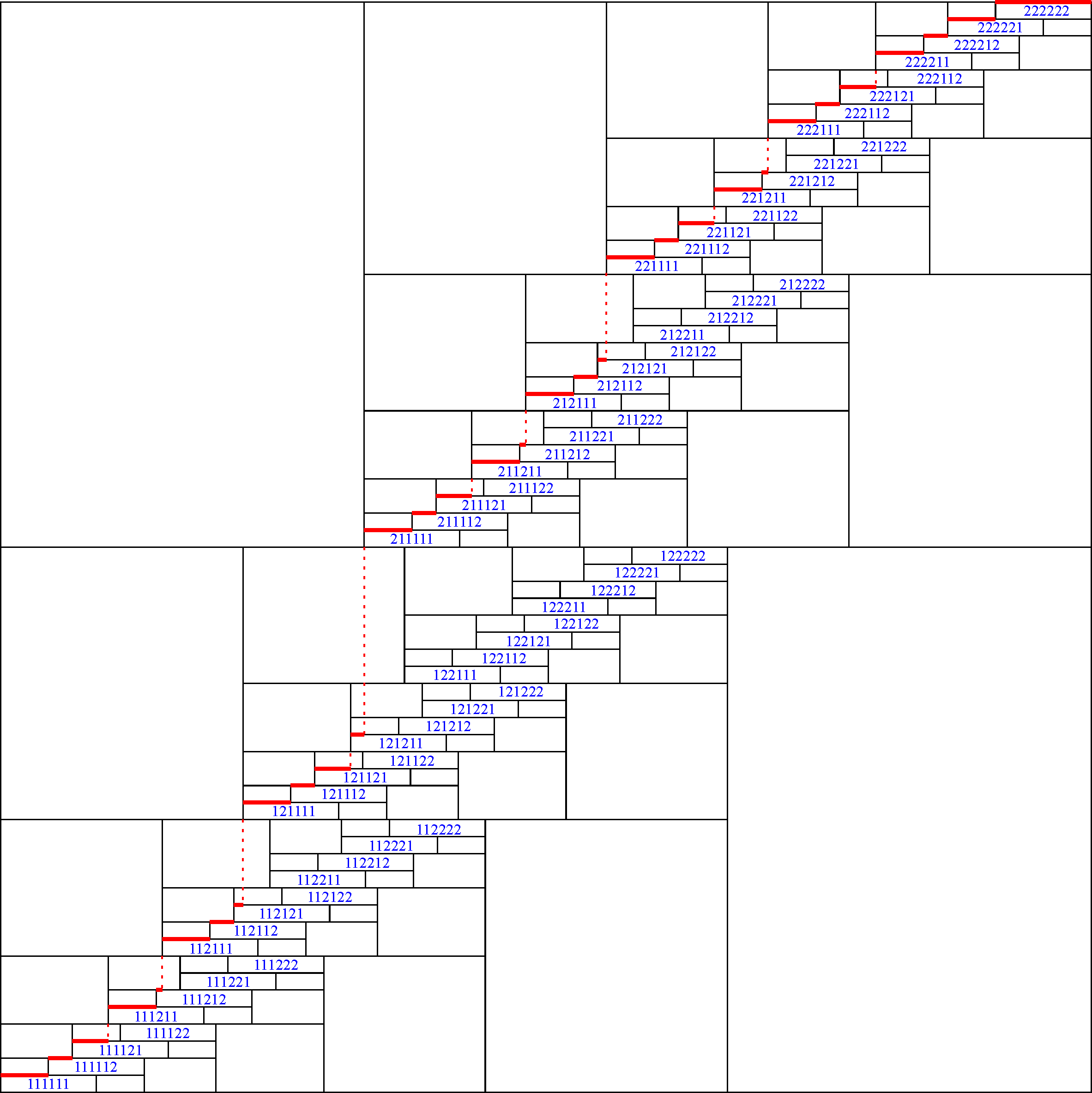}%
\caption{One way of illustrating the top of the attractor of an i.f.s. See Example 3. The ordering here is not lexicographical, so 2 is greater than 1.}%
\label{topaddressessq - copy (2)}%
\end{figure}

EXAMPLE 4 For the i.f.s. $\left\{  \mathbb{R};f_{1}(x)=2x/3;f_{2}%
(x)=1-2x/3\right\}  ,$ each of the strings $\overline{1},\overline
{2},\overline{12},\overline{21}$ belongs to $\Sigma_{top}$ and is reversible.
Figure \ref{topadd}(b) illustrates the sets in $\Sigma_{top,n}$ for
$n=0,1,2,3,4,5$. Here it appears that all addresses are reversible.

\begin{figure}[ptb]%
\centering
\includegraphics[
height=1.548in,
width=4.1589in
]%
{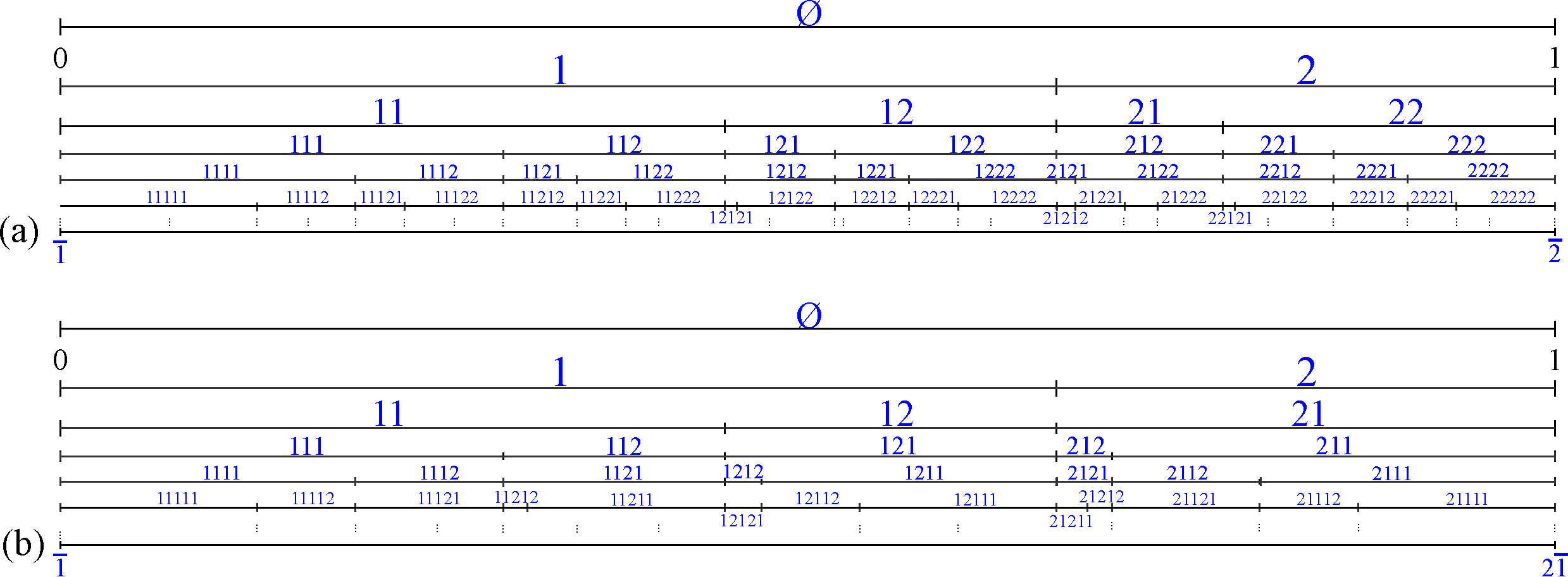}%
\caption{See Examples 3 and 4.}%
\label{topadd}%
\end{figure}

Let us define a \textit{new tile} to be a tile at level $n+1$ that is not
contained in any tile at level $n.$ Also, a \textit{child }or \textit{child
tile}, is a tile at level $n+1$ that is contained in a tile, its
\textit{parent} at level $n.$

\begin{theorem}
\label{strichgentheorem} Let $F$ be an invertible contractive i.f.s. on
$\mathbb{R}^{n}$, as defined in Equation \ref{ifstype}. Let $\mathbf{i\in
}\Sigma_{top}$ be reversible. Each tile in $\Pi_{top}(\mathbf{i}|k+1)$ is
either (i) a nonempty subset, the child\ of a tile in $\Pi_{top}%
(\mathbf{i}|k),$ of the form $tile(i_{1}...i_{k+1}.i_{k+1}p_{1}...p_{k}),$ or
(ii) a nonempty set of the form $tile(i_{1}...i_{k+1}.q_{1}q_{2}%
...q_{k}q_{k+1})$ where $q_{1}\neq i_{k+1},$ a new tile. Each tile in
$\Pi_{top}(\mathbf{i}|k)$ contains exactly one child in $\Pi_{top}%
(\mathbf{i}|k+1)$.
\end{theorem}

\begin{proof}
We can write
\begin{align*}
\Pi_{top}(\mathbf{i}|k+1)  &  =\left\{  tile(i_{1}i_{2}...i_{k+1}.i_{k+1}%
p_{2}...p_{k+1})|i_{k+1}p_{2}...p_{k+1}\in\Sigma_{top,k+1}\right\} \\
&  \cup\left\{  tile(i_{1}i_{2}...i_{k+1}.j_{1}j_{2}...j_{k+1})|j_{1}%
j_{2}...j_{k+1}\in\Sigma_{top,k+1},j_{1}\neq i_{k+1}\right\}
\end{align*}
Each tile in the first set is a subset of a tile in $\Pi_{top}(\mathbf{i}|k),$
and it is non-empty because $\mathbf{i}$ is reversible. (By reversibility, the
set of top addresses $\{i_{k+1}p_{1}p_{2}...p_{k}\in\Sigma_{top,k+1}%
|p_{1}p_{2}...p_{k}\in\Sigma_{top,k}\}$ is nonempty.)

Consider any tile $tile(i_{1}i_{2}...i_{k+1}.p_{1}p_{2}...p_{k+1})$ in the
second set. By Lemma \ref{containlemma}: if $tile(i_{1}i_{2}...i_{k}%
.p_{1}p_{2}...p_{k})\subset$ $tile(i_{1}i_{2}...i_{k+1}.i_{k+1}p_{2}%
...p_{k+1}),$ then $i_{k+1}p_{1}p_{2}...p_{k}=j_{1}j_{2}j_{3}..j_{k+1}$ which
is not possible because $j_{1}\neq i_{k+1}$. So no tile in the second set is
contained in a tile in the first set. That is to say, the tiles in the second
set, which have non-cancelling addresses, are \textquotedblleft
new\textquotedblright\ and do not contain any tile in the first set.

This says that every tile at level $k$ has a unique child at level $k+1,$
either equal to its parent, or smaller but not empty; also, there are new
tiles at level $k+1$ which do not have predecessors at level $k$, because
$\mathcal{A}_{k+1}\neq\mathcal{A}_{k}.$ Each tile in $\Pi_{top}(\mathbf{i}|k)$
contains a child in $\Pi_{top}(\mathbf{i}|k+1)$. One deduces that
$\mathcal{A}_{k+1}\backslash\cup\left\{  \text{children of tiles at level
}k\right\}  $ is tiled by new tiles.
\end{proof}

In the special case $\mathbf{i=}\overline{\mathbf{1}},$ also considered by
Strichartz in Theorem \ref{theorem4}, we have:

\begin{theorem}
\label{onebartheorem}Let $F$ be an invertible contractive i.f.s. on
$\mathbb{R}^{n}$, as defined in Equation \ref{ifstype}. Then $\Pi
_{top}(\overline{\mathbf{1}})$ is a well defined tiling of $\mathcal{A}%
(\overline{\mathbf{1}}):$ specifically $\Pi_{top}(\overline{\mathbf{1}%
}|k)\subset\Pi_{top}(\overline{\mathbf{1}}|k+1),$ and%
\[
\Pi_{top}(\overline{\mathbf{1}})=\bigcup\limits_{k=1}^{\infty}\Pi
_{top}(\overline{\mathbf{1}}|k).
\]
Each tile $\Pi_{top}(\overline{\mathbf{1}}|k)$ (for all $k\in\mathbb{N}$) in
$\Pi_{top}(\overline{\mathbf{1}})$ can be written $tile((\overline{1}%
|k)|t_{1}t_{2}...t_{k})$ for some $t_{1}t_{2}...t_{k}\in\Sigma_{top,k}$ for
some $k,$ with $t_{1}\neq1.$ The tile $A$ corresponds to $k=0.$
\end{theorem}

\begin{proof}
The result follows from the observation that in this case all children are
exact copies of their parents. To see this simply note that $f_{1}^{-1}%
\pi_{top}(1t_{1}t_{2}\dots t_{k})=\pi_{top}(t_{1}t_{2}\dots t_{k})$ for all
$1t_{1}t_{2}\dots t_{k}\in\Sigma_{top,k+1}.$
\end{proof}

For future work, one can consider the case where $A$ is homeomorphic to a
ball. By introducing a stronger notion of reversibility (see also
\cite{tilings}), that requires the tops dynamical system orbit of a reversible
point $\mathbf{i}\in\Sigma_{top}$ to be contained in a compact set $A^{\prime
}$ contained in the interior of $A,$ one can ensure that new tiles are located
further and further away from $A.$ This means that new tiles have only
finitely many successive generations of children (one child at each subsequent
generation) before children are identical to their parents. Hence, given any
ball $B$ of finite radius, the set of tiles in $\Pi_{top}(\mathbf{i}|k)$ that
have nonempty intersection with $B$ remains constant for all large enough $k.$
In such cases one $\Pi_{top}(\mathbf{i}|k)\cap B$ is constant for all $k$
sufficiently large, and so the tiling $\Pi_{top}(\mathbf{i})$ is well defined.
We note that if $\mathbf{i}$ is disjunctive then $\mathcal{A}(\mathbf{i}%
)=\mathbb{R}^{n},$ see \cite{tilings}.

We conjecture that if $A$ is homeomorphic to a ball and if $\mathbf{i\in
\Sigma}_{top}$ is both reversible and disjunctive (relative to the top), then $\Pi_{top}%
(\mathbf{i})$ is a well defined tiling of $\mathbb{R}^{n}.$

\subsection{A leafy example of a two-dimensional top tiling\label{leafysec}}

For a two-dimensional affine transformation $f:\mathbb{R}^{2}\rightarrow
\mathbb{R}^{2}$ we write
\[
\text{ }f=%
\begin{bmatrix}
a & b & e\\
c & d & g
\end{bmatrix}
\text{ for }f(x,y)=(ax+by+e,cx+dy+g)\text{ }%
\]
where $a,b,c,d,e,g\in\mathbb{R}$. We consider the i.f.s. defined by the two
similitudes%
\begin{equation}
f_{1}=%
\begin{bmatrix}
0.7526 & -.2190 & .2474\\
0.2190 & 0.7526 & -.0726
\end{bmatrix}
,f_{2}=%
\begin{bmatrix}
-0.7526 & 0.2190 & 1.0349\\
0.2190 & 0.7526 & 0.0678
\end{bmatrix}
\label{leafifs}%
\end{equation}
The attractor, $L$=leaf, illustrated in Figure \ref{l0}, is made of two
overlapping copies of itself. The copy illustrated in black is associated with
$f_{1}.$ The point with top address $\overline{\mathbf{1}}=111\dots$ is
represented by the tip of the stem of the leaf. The stem is actually arranged
in an infinite spiral, not visible in the picture.
In all tiling pictures, the colors of the tiles were obtained by overlaying the tiling on a colorful photograph: the color of each tile is the color of a point beneath it. In this way, if the tiles were very small, the tiling would look like a mozaic representation of the underlying picture.

\begin{figure}[ptb]%
\centering
\includegraphics[
height=1.1796in,
width=2.5495in
]%
{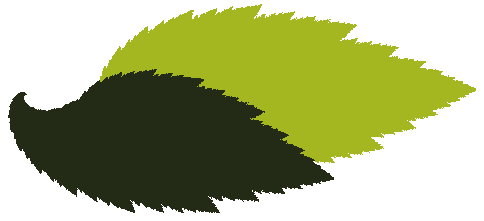}%
\caption{The overlapping attractor of an i.f.s. of two similitudes,
each with the same scaling factor.}%
\label{l0}%
\end{figure}

Figure \ref{overlapleaf} illustrates the top of $L$ at depths $n\in
\{1,2,\dots,6\}$ labelled by the addresses in $\Sigma_{n,top}$ .

\begin{figure}[ptb]%
\centering
\includegraphics[
height=2.0in,
width=5.5in
]%
{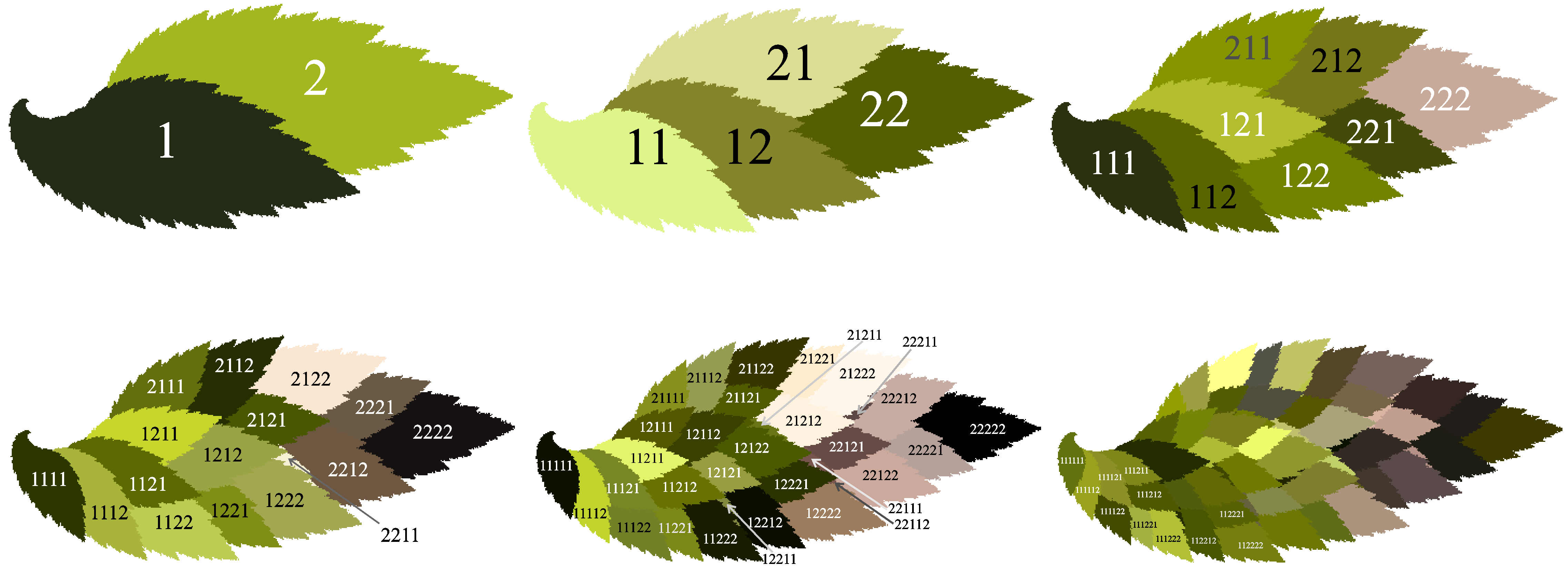}%
\caption{Successsive fractal tops.}%
\label{overlapleaf}%
\end{figure}

Figure \ref{newseq} illustrates the successive blowups $\Pi_{top,n}%
(\overline{1}|n)$ for $n=1,2,\dots,6$ for the i.f.s. in Equation (5.2). See
also Figure \ref{ontopfig} where the successive images are illustrated in
their correct relative positions.

\begin{figure}[ptb]%
\centering
\includegraphics[
height=2.7648in,
width=5.0047in
]%
{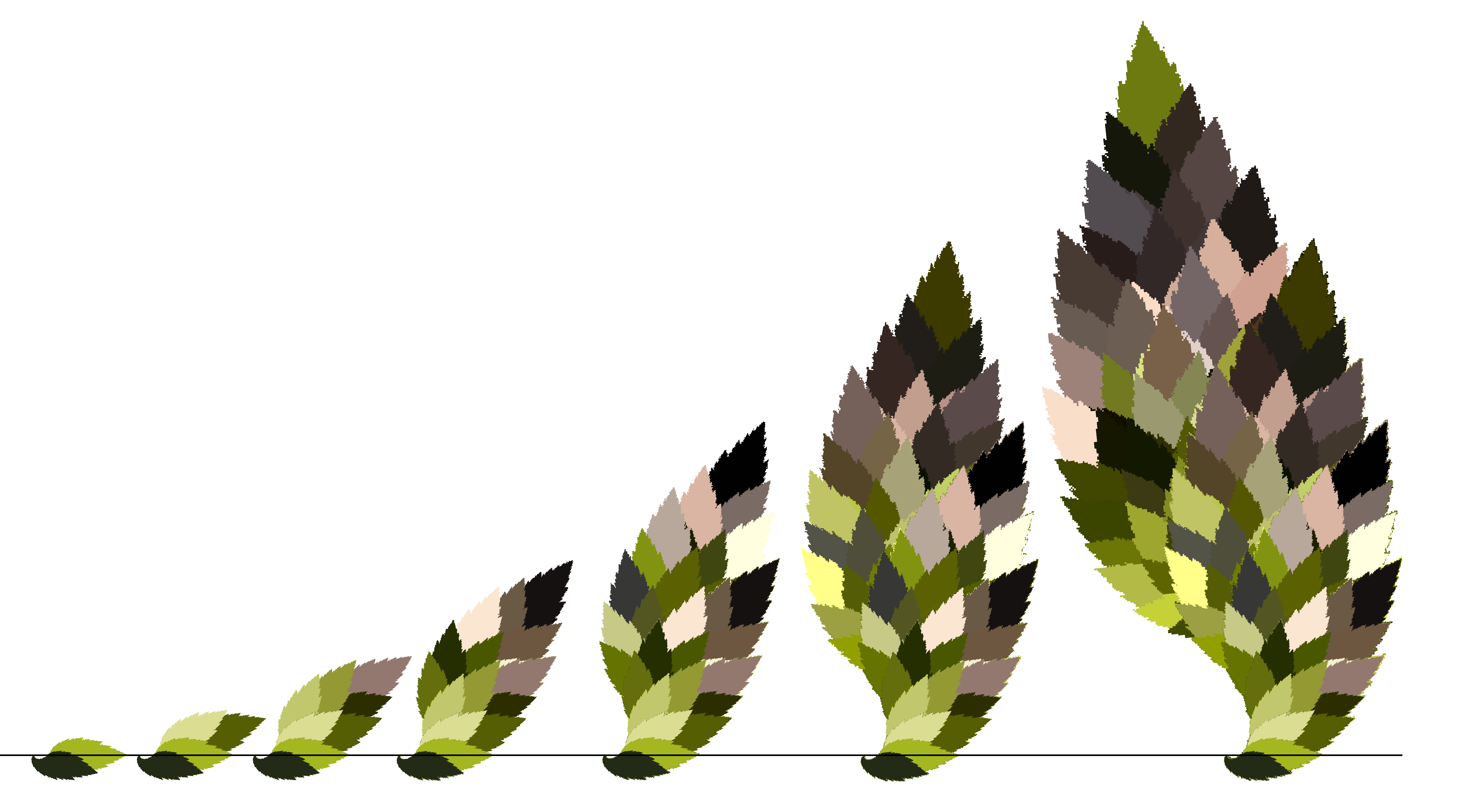}%
\caption{This shows the sequence of tops $\Pi(111...|n)$ for
$n=0,1,\dots,6$ for the leaf i.f.s.  In each case the tip of the stem is at the origin.}%
\label{newseq}%
\end{figure}

\begin{figure}[ptb]%
\centering
\includegraphics[
height=3.44in,
width=2in
]%
{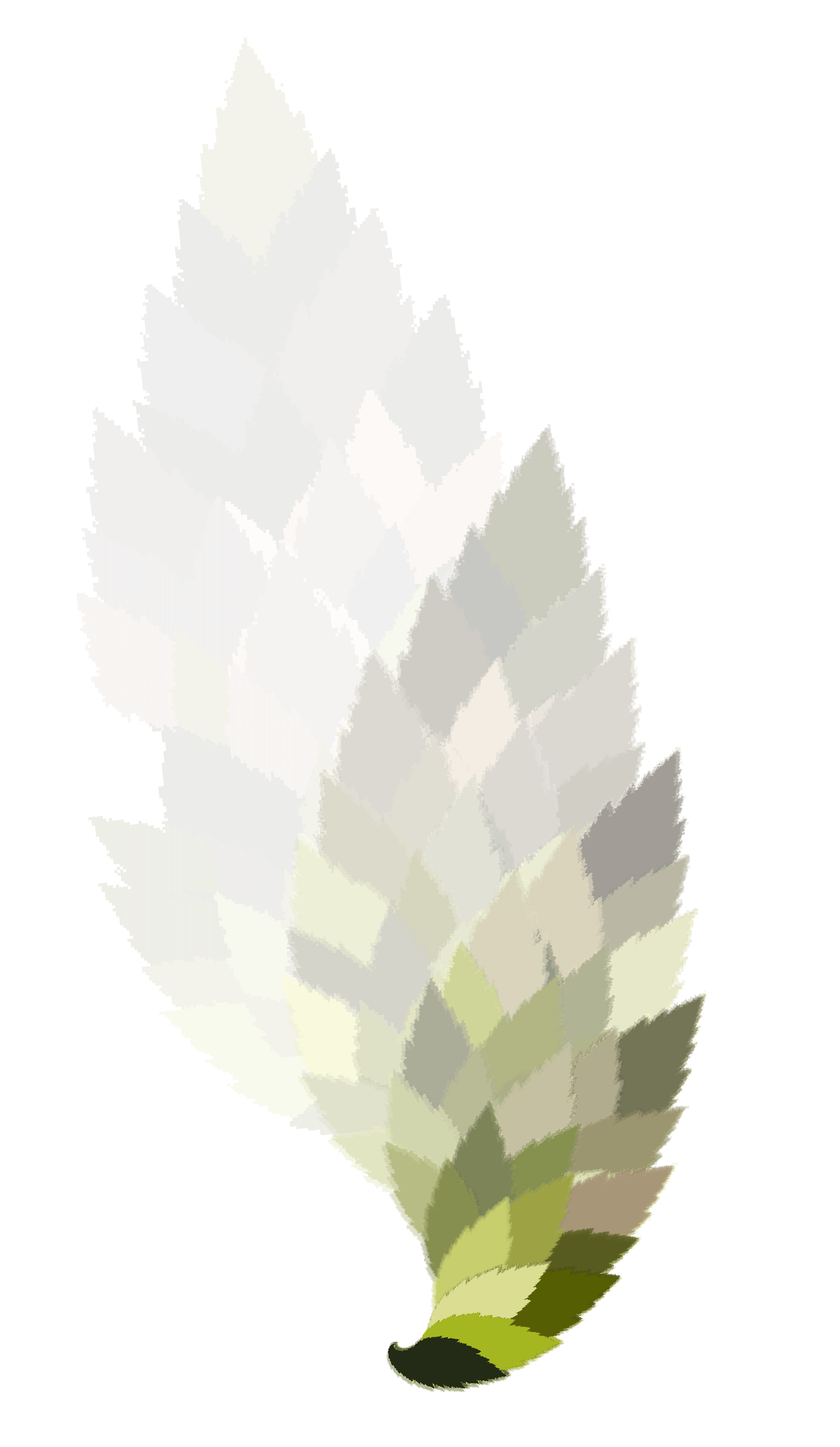}%
\caption{This illustrates the relationship between the successive partial tilings in Figure 8.}%
\label{ontopfig}%
\end{figure}

Figure \ref{strichleaf2} shows a patch of a leaf tiling, illustrating its
complexity. Figure \ref{sandraFern} illustrates a patch of a top tiling
obtained using an i.f.s. of four maps.

\begin{figure}[ptb]%
\centering
\includegraphics[
height=2.765in,
width=3.4203in
]%
{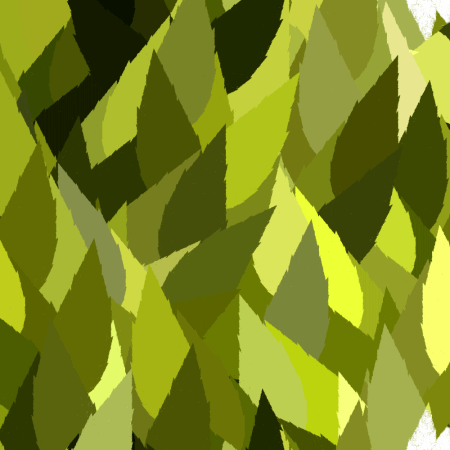}%
\caption{Patch of a leaf tiling.}%
\label{strichleaf2}%
\end{figure}

\begin{figure}[ptb]%
\centering
\includegraphics[
height=2.765in,
width=3.4203in
]%
{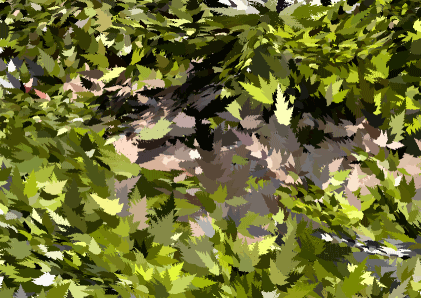}%
\caption{Fractal top for an i.f.s of three maps looks both random and somewhat natural, but is not the real thing, compare with Figure 1.}%
\label{sandraFern}%
\end{figure}

ACKNOWLEDGEMENTS: We thank both Brendan Harding and Giorgio Mantica for careful reading and corrections.

\end{document}